\documentclass[a4paper]{article}
\usepackage{amsmath,amssymb,amsthm}

\newcommand{\eij}{e_{ij}}

\newcommand{\epq}{e_{pq}}

\newcommand{\mpq}{m_{pq}}

\newcommand{\one}{\mathbf{1}}
\newcommand{\lex}[1]{{ \,\underset{\mathrm{lex}}{#1}\, }}

\newcommand{\surj}{\twoheadrightarrow}

\newcommand{\mor}{\rightarrow}
\newcommand{\schub}{\mathfrak{S}}
\newcommand{\smod}{\mathcal{S}}
\newcommand{\der}{\partial}
\newcommand{\borel}{\mathfrak{b}}
\newcommand{\nplus}{\mathfrak{n}^+}
\newcommand{\csa}{\mathfrak{h}}
\newcommand{\ualg}{\mathcal{U}}

\newcommand{\inv}{\mathrm{inv}}

\newcommand{\ch}{\mathrm{ch}}
\newcommand{\Ext}{\mathrm{Ext}}
\newcommand{\Hom}{\mathrm{Hom}}
\newcommand{\NN}{\mathbb{N}}
\newcommand{\nonneg}{\ZZ_{\geq 0}}
\newcommand{\ZZ}{\mathbb{Z}}

\newcommand{\lmb}{\lambda}

\newtheorem{lem}{Lemma}[section]
\newtheorem{prop}[lem]{Proposition}
\newtheorem{thm}[lem]{Theorem}
\newtheorem{cor}[lem]{Corollary}

\newtheorem*{thm*}{Theorem}
\theoremstyle{definition}
\newtheorem{rmk}[lem]{Remark}
\newtheorem{eg}[lem]{Example}

\title{Tensor product of Kra\'skiewicz and Pragacz's modules}
\author{
Masaki Watanabe \\ 
Graduate School of Mathematical Sciences, The University of Tokyo, \\
3-8-1 Komaba Meguro-ku Tokyo 153-8914, Japan \\ \texttt{mwata@ms.u-tokyo.ac.jp}}
\date{\empty}

\begin{document}
\maketitle

\noindent\textbf{Abstract. } This paper explores further properties of modules related with Schubert polynomials, introduced by Kra\'skiewicz and Pragacz. 
In this paper we show that any tensor product of Kra\'skiewicz-Pragacz modules admits a filtration by Kra\'skiewicz-Pragacz modules. This result can be seen as a module-theoretic counterpart of a classical result that the product of Schubert polynomials is a positive sum of Schubert polynomials. 

\section{Introduction}
The study of Schubert polynomials is an important and interesting subject in algebraic combinatorics. 
One of the possible methods for studying Schubert polynomials is through the modules introduced by Kra\'skiewicz and Pragacz (\cite{KP}). 
Let $\borel$ be the Lie algebra of all upper triangular matrices. 
For each permutation $w$, Kra\'skiewicz and Pragacz defined a representation  $\smod_w$ of $\borel$ such that its character with respect to the subalgebra $\csa$ of all diagonal matrices is equal to the Schubert polynomial $\schub_w$. 
In the author's previous paper \cite{W}, the author investigated the characterization of modules having a filtration with subquotients isomorphics to Kra\'skiewicz-Pragacz modules (or a \textit{KP filtration} for short). One of the motivation for studying such class of modules is a study of Schubert-positivity: if a module $M$ has a KP filtration, then it follows that its character $\ch(M)$ is a positive sum of Schubert polynomials. 

One of the positivity properties of Schubert polynomials is that the product $\schub_w\schub_v$ of Schubert polynomials is a positive sum of Schubert polynomials. 
In this paper we show that the module-theoretic counterpart of this fact holds.
The main result of this paper is the following theorem (Theorem \ref{mainthm}):
\begin{thm*}
For any $w, v \in S_\infty^{(n)}$, the module $\smod_w \otimes \smod_v$ has a KP filtration.
\end{thm*}
Not only finding a module-theoretic counterpart to the classical result, we also present a new interesting positivity result on Schubert polynomials (Corollary \ref{pleth}) using our main result. 

The paper is organized as follows. 
In section \ref{prelimi}, we prepare some basic facts about Schubert polynomials and KP modules, as well as some results from the author's previous paper \cite{W}. 
In section \ref{monkcase}, we prove the special case of our main theorem corresponding to Monk's formula for Schubert polynomials: 
we show that $\smod_w \otimes \smod_{s_\nu}$ has a KP filtration for any $w \in S_\infty^{(n)}$, $1 \leq \nu \leq n-1$. 
The problem asking the existence of such filtration
was first raised by Kra\'skiewicz and Pragacz \cite[Remark 5.2]{KP}, 
and our result gives the answer to this problem. 
In section \ref{generalcase} we deal with the general case and prove the main theorem, 
and present some consequences of our result. 

\section{Preliminaries}
\label{prelimi}
Let $\NN$ be the set of all positive integers and let $\nonneg$ be the set of all nonnegative integers. 
A \textit{permutation} $w$ is a bijection from $\NN$ to itself which fixes all but finite points. 
Let $S_\infty$ denote the group of all permutations. 
Let $n$ be a positive integer. 
Let $S_n =\{w \in S_\infty : w(i)=i \; (i>n)\}$
and $S_\infty^{(n)}=\{w \in S_\infty : w(n+1)<w(n+2)<\cdots\}$. 
For $i < j$, let $t_{ij}$ denote the permutation which exchanges $i$ and $j$ and fixes all other points. 
Let $s_i=t_{i,i+1}$. 
For a permutation $w$, let $\ell(w)=\#\{i<j : w(i)>w(j)\}$. 
Let $w_0 \in S_n$ be the longest element of $S_n$, i.e. $w_0(i)=n+1-i$ ($1 \leq i \leq n$). 
For $w \in S_\infty^{(n)}$ we define $\inv(w)=(\inv(w)_1, \ldots, \inv(w)_n) \in \nonneg^n$
by $\inv(w)_i=\#\{j : i<j, w(i)>w(j)\}$. 

For a polynomial $f=f(x_1, \ldots, x_n)$ and $1 \leq i \leq n-1$, we define $\der_if=\frac{f-s_if}{x_i-x_{i+1}}$. 
For $w \in S_\infty^{(n)}$ we can assign its \textit{Schubert polynomial} $\schub_w \in \ZZ[x_1, \ldots, x_n]$, which is recursively defined by 
\begin{itemize}
\item $\schub_{w}=x_1^{w(1)-1}x_2^{w(2)-1} \cdots x_n^{w(n)-1}$ if $w(1)>w(2)>\cdots>w(n)$, and
\item $\schub_{ws_i}=\der_i\schub_w$ if $\ell(ws_i)<\ell(w)$. 
\end{itemize}
Note that the sets $\{\schub_w : w \in S_\infty^{(n)}\}$ and $\{\schub_w : w \in S_n\}$
constitute $\ZZ$-bases for the spaces $\ZZ[x_1, \ldots, x_n]$
 and $H_N=\bigoplus_{0 \leq a_i \leq N-i} \ZZ x_1^{a_1} \cdots x_{N-1}^{a_{N-1}}$ respectively (\cite[(4.11)]{Mac}). 

Schubert polynomials satisfy the following identity called Monk's formula: 
\begin{prop}[{{\cite[$(4.15'')$]{Mac}}}]
Let $w \in S_\infty$ and $\nu \in \NN$. 
Then $\schub_w \schub_{s_\nu} = \sum \schub_{wt_{pq}}$ where the sum is over all $(p,q)$ such that $p \leq \nu < q$ and $\ell(wt_{pq})=\ell(w)+1$. 
\label{polymonk}
\end{prop}

Schubert polynomials also satisfy the following Cauchy identity: 
\begin{prop}[{{\cite[$(5.10)$]{Mac}}}]
$\sum_{w \in S_n} \schub_w(x) \schub_{ww_0}(y) = \prod_{i+j \leq n} (x_i+y_j)$. 
\end{prop}

Let $K$ be a field of characteristic zero. 
Let $\borel=\borel_n$ be the Lie algebra of all upper triangular $K$-matrices. 
and let $\csa \subset \borel$ and $\nplus \subset \borel$ be the subalgebra of all diagonal matrices
and the subalgebra of all strictly upper triangular matrices respectively. 
Let $\ualg(\borel)$ and $\ualg(\nplus)$ be the universal enveloping algebras of $\borel$ and $\nplus$ respectively. 
For a $\ualg(\borel)$-module $M$ and $\lmb = (\lmb_1, \ldots, \lmb_n) \in \ZZ^n$, 
let $M_\lmb = \{m \in M : hm=\langle \lmb,h \rangle m \;\text{($\forall h \in \csa$)}\}$ where $\langle \lmb,h \rangle = \sum \lmb_i h_i$. 
If $M_\lmb \neq 0$ then $\lmb$ is said to be a \textit{weight} of $M$. 
If $M=\bigoplus_{\lmb \in \ZZ^n} M_\lmb$ and each $M_\lmb$ has a finite dimension, 
then $M$ is said to be a \textit{weight $\borel$-module}
and we define $\ch(M)=\sum_{\lmb} \dim M_\lmb x^\lmb$ where $x^\lmb=x_1^{\lmb_1} \cdots x_n^{\lmb_n}$. 
From here we only consider weight $\borel$-modules, and write $\Ext^i$ to mean $\Ext$ functors in the category of all weight $\borel$-modules.  
%Let $\nplus \in \borel$ be the Lie subalgebra of all strictly upper triangular matrices $(a_{ij})_{\substack{i,j \in \NN \\ i<j}}$.
For $1 \leq i<j \leq n$, let $\eij \in \borel$ be the matrix with $1$ at the $(i,j)$-position and all other coordinates $0$. 

For $\lmb \in \ZZ^n$, let $K_\lmb$ denote the one-dimensional $\ualg(\borel)$-module where $h \in \csa$ acts by $\langle \lmb,h \rangle$ and $e_{ij}$ acts by $0$. 
Note that every finite-dimensional weight $\borel$-module admits a filtration by these one dimensional modules. 

In \cite{KP}, Kra\'skiewicz and Pragacz defined certain $\ualg(\borel)$-modules which here we call \textit{Kra\'skiewicz-Pragacz modules} or \textit{KP modules}. 
Here we use the following definition. 
Let $w \in S_\infty^{(n)}$. 
Let $K^n=\bigoplus_{1 \leq i \leq n} K u_i$ be the vector representation of $\borel$. 
For each $j \in \NN$, let $\{i < j : w(i)>w(j)\}=\{i_1, \ldots, i_{l_j}\}$ ($i_1<\cdots<i_{l_j}$), 
and let $u_w^{(j)}=u_{i_1} \wedge \cdots \wedge u_{i_{l_j}} \in \bigwedge^{l_j} K^n$. 
Let $u_w=u_w^{(1)} \otimes u_w^{(2)} \otimes \cdots \in \bigotimes_{j \in \NN} \bigwedge^{l_j} K^n$. 
Then the KP module $\smod_w$ associated to $w$ is defined as $\smod_w=\ualg(\borel)u_w=\ualg(\nplus)u_w$. 

\begin{eg}
If $w=s_i$, then $u_w=u_i$ and thus we have $\smod_w=K^i := \bigoplus_{1 \leq j \leq i} Ku_j$. 
\end{eg}

KP modules satisfy the following: 
\begin{thm}[{{\cite[Remark 1.6, Theorem 4.1]{KP}}}]
For any $w \in S_\infty^{(n)}$, $\smod_w$ is a weight $\borel$-module and $\ch(\smod_w)=\schub_w$. 
\end{thm}

Note that, in the notation of the definition above, 
for $1 \leq p < q \leq n$, the number of $j \in \NN$ such that $e_{pq}u_w^{(j)} \neq 0$
is given by $m_{pq}(w)=\#\{r>q : w(p)<w(r)<w(q)\}$. 
In particular, $e_{pq}^{m_{pq}(w)+1}u_w=0$. 

In the author's previous paper \cite{W} we showed the following criterion for a module to have a filtration by KP modules, 
using the theory of highest-weight categories: 

\begin{thm}[{{\cite[Corollary 6.6, Theorem 7.1]{W}}}]
%Let $\Lmb^{(n)}=\{(a_1, \ldots, a_n) \in \ZZ^n : 0 \leq a_i \leq n-i\}$. 
%Let $M$ be a finite-dimensional weight $\borel$-module whose weight are all in $\Lmb^{(n)}$. 
Let $M$ be a finite-dimensional weight $\borel$-module whose weight are all in $\nonneg^n$. 
Then the followings are equivalent: 
\begin{enumerate}
\item $M$ has a filtration such that each of its subquotients is isomorphic to some $\smod_w$ $(w \in S_\infty^{(n)})$. 
In such case we say that $M$ has a \textit{KP filtration} hereafter. 
\item $\Ext^i(M, \smod_w^* \otimes K_{\rho+k\one})=0$ for all $i \geq 1$, $w \in S_\infty^{(n)}$ and $k \in \ZZ$, 
where $\rho = (n-1, n-2, \ldots, 0)$ and $\one=(1, 1, \ldots, 1)$. 
\item $\Ext^1(M, \smod_w^* \otimes K_{\rho+k\one})=0$ for all $w \in S_\infty^{(n)}$ and $k \in \ZZ$, 
where $\rho = (n-1, n-2, \ldots, 0)$ and $\one=(1, 1, \ldots, 1)$. 
\end{enumerate}
\end{thm}

From this, we obtain the following corollary. 

\begin{cor}
\begin{enumerate}
\item If $M=M_1 \oplus \cdots \oplus M_r$, then $M$ has a KP filtration if and only if each $M_i$ has. 
\item If $0 \mor L \mor M \mor N \mor 0$ is exact and $M$ and $N$ have KP filtrations, 
then $L$ also has a KP filtration. 
\end{enumerate}
\label{filtrlem}
\end{cor}
\begin{proof}
(1) is easy since $\Ext^1(M, N)=\bigoplus \Ext^1(M_i,N)$ for any $N$. 
(2) also follows from the previous theorem, since if $0 \mor L \mor M \mor N \mor 0$ is an exact sequence 
then we have an exact sequence $\Ext^1(M, A) \mor \Ext^1(L, A) \mor \Ext^2(N, A)$ for any $A$. 
\end{proof}

\section{Monk's formula and KP filtration}
\label{monkcase}
In this and the next section we show that the tensor product of KP modules have a KP filtration. 
In this section we deal with the special case that one of the KP modules is $\smod_{s_\nu}$, 
%This corresponds to Monk's formula for Schubert polynomial which describes the expansion of product $\schub_w \schub_{s_\nu}$ into a sum of Schubert polynomials. 
which corresponds to the Monk's formula (Proposition \ref{polymonk}). 
\begin{prop}
Let $w \in S_\infty^{(n)}$ and let $1 \leq \nu \leq n-1$. 
Then $\smod_w \otimes \smod_{s_\nu}$ has a KP filtration. 
\label{monk}
\end{prop}

\begin{rmk}
The case $\nu=1$ was proved by Kra\'skiewicz and Pragacz in \cite[Proposition 5.1]{KP}. 
They raised a problem asking the existence of a filtration corresponding to Monk's formula for general case in \cite[Remark 5.2]{KP}. 
The proposition above gives the answer to this problem. 
\end{rmk}

\begin{lem}
Let $X=\{(p,q) : p \leq \nu < q, \ell(wt_{pq})=\ell(w)+1\}$. 
For $(p,q) \in X$, define $v_{pq}=(\epq^{\mpq(w)} u_w) \otimes u_p \in \smod_w \otimes \smod_{s_\nu}$, 
where $\mpq(w)=\#\{r > q : w(p)<w(r)<w(q)\}$ as defined in the previous section. 
Note that $v_{pq}=\epq^{\mpq(w)}(u_w \otimes u_p)=(\text{const.}) \cdot \epq^{\mpq(w)+1}(u_w \otimes u_q)$. 
Then:
\begin{enumerate}
\item $\{ v_{pq} : (p,q) \in X \}$ generates $\smod_w \otimes \smod_{s_\nu}$ as an $\ualg(\borel)$-module. 
\item For each $(p,q) \in X$, there exists a $\borel$-homomorphism $\phi_{pq} : \smod_w \otimes \smod_{s_\nu} \mor
\bigotimes_j \left( \bigwedge^{l_j(wt_{pq})} K^n \right)$ 
where $l_j(wt_{pq})=\#\{i < j : wt_{pq}(i)>wt_{pq}(j)\}$, 
such that: 
\begin{enumerate}
\item $\phi_{pq}(K v_{pq}) = K u_{wt_{pq}}$, and
\item for $(p', q') \in X$, $\phi_{pq}(v_{p'q'}) \neq 0$ implies $w(p') \leq w(p)$ and $w(q') \leq w(q)$. 
\end{enumerate}
\end{enumerate}
\label{monklem}
\end{lem}

First we see that Proposition \ref{monk} follows from this lemma. 
Index the elements of $X$ as $X=\{ (p_0, q_0), \ldots, (p_{r-1}, q_{r-1}) \}$ 
so that there exists no $i<j$ such that $w(p_j) \leq w(p_i)$ and $w(q_j) \leq w(q_i)$ hold simultaneously. 
For $0 \leq i \leq r$, let $F_i$ be the submodule of $\smod_w \otimes \smod_{s_\nu}$ generated by $v_{p_j,q_j}$ ($j \geq i$). 
Then $F_r=0$ and, by (1) of the lemma, $F_0=\smod_w \otimes \smod_{s_\nu}$. 
Moreover, by (2) of the lemma, 
$\phi_{p_i, q_i}$ induces a surjective morphism $F_i/F_{i+1} \surj \smod_{wt_{p_i, q_i}}$. 
But since by Monk's formula we have
$\dim(\smod_w \otimes \smod_{s_\nu}) = \sum \dim \smod_{wt_{p_i, q_i}}$, 
all of these morphisms must be isomorphisms and hence the theorem follows. 

Let us now prove Lemma \ref{monklem}. 
\begin{proof}
(1): 
Let $M$ be the submodule of $\smod_w \otimes \smod_{s_\nu}$ 
generated by $\{ v_{pq} : (p,q) \in X \}$. 
If we show that $u_w \otimes u_p \in M$ for all $p \leq \nu$, 
then we can show that $xu_w \otimes u_p \in M$ for any $x \in \ualg(\nplus)$ and any $p \leq \nu$ by induction on $p$, 
 since $x(u_w \otimes u_p) \in xu_w \otimes u_p + \smod_w \otimes \left( \bigoplus_{p'<p} Ku_{p'} \right)$. 

Let $p \leq \nu$. We want to show that $u_w \otimes u_p \in M$. 
Define a sequence $r_0, r_1, \ldots, r_k$ by the following algorithm: 
\begin{itemize}
\item Let $r_0=p$. 
\item For $j=1, 2, \ldots$, 
take $r_j$ so that $r_j > r_{j-1}, w(r_j) > w(r_{j-1})$, 
and $w(r_j)$ is as small as possible among such. If $r_j > \nu$, let $k=j$ and stop. 
\end{itemize}
By the construction we have $(r_{k-1},r_{k}) \in X$ and $m_{r_{k-1},r_k}(w)=0$, 
and thus $u_w \otimes u_{r_{k-1}} \in M$. 
If $k=1$ we are done since $r_{k-1}=p$. 
Consider the case $k \geq 2$. 
By the construction, we have $m_{p,r_{k-1}}(w)=0$. 
Thus $e_{p,r_{k-1}}u_w=0$ and $u_w \otimes u_p = e_{p,r_{k-1}}(u_w \otimes u_{r_{k-1}}) \in M$. 
This finishes the proof of (1). 

(2):
Let $(p,q) \in X$. 
Let $a=\#\{r < p : w(r)>w(p)\}=l_p(w), b=\#\{r < p : w(r)>w(q)\}$ and $c=\#\{p<r<q : w(r)>w(q)\}=l_q(w)-b$. 
Then $l_p(wt_{pq})=b$, $l_q(wt_{pq})=a+c+1$ and $l_j(wt_{pq})=l_j(w)$ for $j \neq p,q$. 
Let $\tilde{\phi}_{pq}$ be the $\borel$-homomorphism defined on $\smod_w \otimes K^n$ by the composition of the following morphisms: 
\begin{align*}
\smod_w \otimes K^n &= \bigotimes_j \big( \bigwedge \!{}^{l_j(w)} K^n \big) \otimes K^n \\
&= \bigotimes_{j \neq p,q} \big( \bigwedge \!{}^{l_j(w)} K^n \big) \otimes \bigwedge \!{}^{a} K^n \otimes \bigwedge \!{}^{b+c} K^n \otimes K^n \\
&\mor \bigotimes_{j \neq p,q} \big( \bigwedge \!{}^{l_j(w)} K^n \big) \otimes \bigwedge \!{}^{a} K^n \otimes \bigwedge \!{}^{b} K^n \otimes \bigwedge \!{}^{c} K^n \otimes K^n \\
&= \bigotimes_{j \neq p,q} \big( \bigwedge \!{}^{l_j(w)} K^n \big) \otimes \bigwedge \!{}^{a} K^n \otimes \bigwedge \!{}^{c} K^n \otimes K^n \otimes \bigwedge \!{}^{b} K^n \\
&\mor \bigotimes_{j \neq p,q} \big( \bigwedge \!{}^{l_j(w)} K^n \big) \otimes \bigwedge \!{}^{a+c+1} K^n \otimes \bigwedge \!{}^{b} K^n \\
&= \bigotimes_j \big( \bigwedge \!{}^{l_j(wt_{pq})} K^n \big)
\end{align*}
where the maps $\bigwedge^{b+c} K^n \mor \bigwedge^{b} K^n \otimes \bigwedge^{c} K^n$
and $\bigwedge^{a} K^n \otimes \bigwedge^{c} K^n \otimes K^n \mor \bigwedge^{a+c+1} K^n$ are the natural ones. 
We define $\phi_{pq}$ as the restriction of $\tilde{\phi}_{pq}$ to $\smod_w \otimes \smod_{s_\nu}$. 

For a subset $I = \{i_1 < \cdots < i_a\} \subset \{1, \ldots, n\}$, we write $u_I=u_{i_1} \wedge \cdots \wedge u_{i_a}$. 
Then $u_w=u^{\neq} \otimes u_A \otimes (u_B \wedge u_C)$
where $A=\{r < p : w(r)>w(p)\}$, $B=\{r<p : w(r)>w(q)\}$, $C=\{p<r<q : w(r)>w(q)\}$
and $u^{\neq} = \bigotimes_{j \neq p,q} u_{\{r < j : w(r)>w(j)\}}$. 
Note that $A \supset B$. 
Note also that $u_{wt_{pq}}$ is a multiple of $\epq^{\mpq(w)} (u^{\neq} \otimes (u_A \wedge u_C \wedge u_p) \otimes u_B)$. 

For any $1 \leq r \leq n$, the image of $u_w \otimes u_r \in \smod_w \otimes K^n$ under the morphism $\tilde{\phi}_{pq}$ can be computed as follows: 
\begin{align*}
u_w \otimes u_r &= u^{\neq} \otimes u_A \otimes (u_B \wedge u_C) \otimes u_r \\
&\mapsto u^{\neq} \otimes u_A \otimes \left( \sum_{\substack{|B'|=b, |C'|=c \\ B \sqcup C = B' \sqcup C'}} \pm u_{B'} \otimes u_{C'} \right) \otimes u_r \\
&\mapsto u^{\neq} \otimes \left( \sum_{\substack{|B'|=b, |C'|=c \\ B \sqcup C = B' \sqcup C'}} \pm (u_{A} \wedge u_{C'} \wedge u_r) \otimes u_{B'} \right). 
\end{align*}
But since $A \supset B$, $u_A \wedge u_{C'} \wedge u_r$ vanishes if $C' \cap B \neq \varnothing$. 
So the only term in the last sum which can be nonzero is $\pm (u_A \wedge u_C \wedge u_r) \otimes u_B$. 
Therefore we have $\tilde{\phi}_{pq}(u_w \otimes u_r) = \pm u^{\neq} \otimes (u_A \wedge u_C \wedge u_r) \otimes u_B$. 

Now we see that $\phi_{pq}(v_{pq})=\phi_{pq}(\epq^{\mpq(w)} (u_w \otimes u_p))=\epq^{\mpq(w)} \phi_{pq}(u_w \otimes u_p)
=\pm \epq^{\mpq(w)}(u^{\neq} \otimes (u_A \wedge u_C \wedge u_p) \otimes u_B)$ is a nonzero multiple of $u_{wt_{pq}}$. 
Thus (a) follows. 

Now consider another $(p', q') \in X$. Since $v_{p'q'}$ is a multiple of $e_{p'q'}^{m_{p'q'}(w)+1} (u_w \otimes u_{q'})$, 
$\phi_{pq}(v_{p'q'})=(\text{const.}) \cdot e_{p'q'}^{m_{p'q'}(w)+1}\tilde\phi_{pq}(u_w \otimes u_{q'})
=(\text{const.}) \cdot e_{p'q'}^{m_{p'q'}(w)+1} (u^{\neq} \otimes (u_A \wedge u_C \wedge u_{q'}) \otimes u_B)$. 
In order this to be nonzero, we must have $e_{p'q'}(u_A \wedge u_C \wedge u_{q'}) \neq 0$, 
since $u^{\neq}$ is annihilated by $e_{p'q'}^{m_{p'q'}(w)+1}$ and $u_B$ is annihilated by $e_{p'q'}$. 
Thus if $\phi_{pq}(v_{p'q'}) \neq 0$, then $p', q' \not\in A \cup C$. 

Note that, for an $r \in \{1, \ldots, n\}$, $r \in A \cup C$ if and only if $r<q$ and $w(r)>w(p)$
(we used here the fact that there exists no $p<r<q$ with $w(p)<w(r)<w(q)$). 
Using this and the fact $p' \leq \nu < q$, we see that $w(p') \leq w(p)$. 
Moreover, it follows that $w(q') \leq w(q)$ unless $q'>q$,  
but if $q'>q$ and $w(q')>w(q)$ then $\ell(wt_{p'q'}) > \ell(w)+1$
since $p'<q<q'$ and $w(p')<w(q)<w(q')$, and thus such a case cannot occur. 
Therefore we have (b). 
\end{proof}
\nocite{*}

\section{Tensor product of KP modules}
\label{generalcase}
\begin{thm}
For any $w \in S_n$ and $v \in S_\infty^{(n)}$, $\smod_w \otimes \smod_v$ has a KP filtration. 
\label{mainthm}
\end{thm}
Thus if we let $n \rightarrow \infty$, we see that for any $w, v \in S_\infty$, the module
(over the Lie algebra $\borel_\infty$ of upper triangular matrices of infinite size) $\smod_w \otimes \smod_v$ has a filtration by KP modules, 
if we define $\smod_w$ appropriately for all $w \in S_\infty$ as a module over $\borel_\infty$ 
(in particular, we see that the theorem in fact holds for any $w, v \in S_\infty^{(n)}$). 

In order to prove this theorem, we begin with some observations. 
For a $w \in S_n$, we define a $\borel$-module $T_w = \bigotimes_{2 \leq i \leq n} \big( \bigwedge^{l_i(w)} K^{i-1} \big)$, 
where $l_i(w)=\#\{j<i : w(j)>w(i)\}$ as before. 
Since $T_w$ is a direct sum component of $\bigotimes_{2 \leq i \leq n} \bigotimes^{l_i(w)} K^{i-1}=\bigotimes_{2 \leq i \leq n} \bigotimes^{l_i(w)} \smod_{s_{i-1}}$, 
$T_w$ have a KP filtration by Proposition \ref{monk} and Corollary \ref{filtrlem}(1). 
We show the following lemma:
\begin{lem}
Let $w \in S_n$. 
Then there exists an exact sequence $0 \mor \smod_w \mor T_w \mor N \mor 0$
such that $N$ has a filtration whose each subquotient is isomorphic to some $\smod_u$ $(\text{$u \in S_n$, $u^{-1} \lex> w^{-1}$})$. 
Here for permutations $x$ and $y$, $x \lex> y$ if there exists an $i$ such that $x(1)=y(1), \ldots, x(i-1)=y(i-1), x(i)>y(i)$. 
\label{tillem}
\end{lem}

We see first that the theorem easily follows from this lemma by induction on $w$. 
From the lemma we get an exact sequence
$0 \mor \smod_w \otimes \smod_v \mor T_w \otimes \smod_v \mor N \otimes \smod_v \mor 0$. 
Here $T_w \otimes \smod_v$ have a KP filtration by Proposition \ref{monk} and Corollary \ref{filtrlem}(1), 
since it is a direct sum component of $\left( \bigotimes_{2 \leq i \leq n} \bigotimes^{l_i(w)} \smod_{s_{i-1}} \right) \otimes \smod_v$. 
Moreover $N \otimes \smod_v$ have a KP filration by the induction hypothesis. 
Hence the claim follows from Corollary \ref{filtrlem}(2).

To prove Lemma \ref{tillem}, we need a result from \cite{W}. 
In \cite[Proposition 5.4]{W} the author showed that, for $x,y \in S_\infty^{(n)}$: 
\begin{itemize}
\item if $(\smod_x)_{\inv(y)} \neq 0$ (i.e. the coefficient of $x^{\inv(y)}$ in $\schub_x$ is nonzero) then $y^{-1} \lex\geq x^{-1}$, and
\item if $\Ext^1(\smod_x,K_{\inv(y)}) \neq 0$ then $y^{-1} \lex< x^{-1}$. 
\end{itemize}
In particular, if $w, u \in S_\infty^{(n)}$ and $\Ext^1(\smod_w, \smod_u) \neq 0$, 
then there exists an $x \in S_\infty^{(n)}$ such that $(\smod_u)_{\inv(x)} \neq 0$ and $\Ext^1(\smod_w, K_{\inv(x)}) \neq 0$, 
and thus $u^{-1} \lex\leq x^{-1} \lex< w^{-1}$. 

Let us now prove Lemma \ref{tillem}. 
\begin{proof}
Let $l_i=l_i(w)$ and let the integers $n_{wu} \in \ZZ$
be defined by $\prod_{2\leq i \leq n} e_{l_i}(x_1, \ldots, x_{i-1}) = \sum_{x \in S_n} n_{wu} \schub_u$ 
where $e_k$ denotes the $k$-th elementary symmetric polynomial. 
Since the left-hand side is the character of $T_w$, 
the number $n_{wu}$ is the number of times $\smod_u$ appears as a subquotient in (any) KP filtration of $T_w$. 

Since 
$\sum_{u \in S_n} \schub_u(x) \schub_{uw_0}(y)=\prod_{i+j\leq n} (x_i+y_j)
=\sum_{0 \leq a_i \leq n-i} \big( \prod_{1 \leq i \leq n-1} x_i^{n-i-a_i} \cdot \prod_{1 \leq i \leq n-1} e_{a_i}(y_1, \ldots, y_{n-i}) \big)$, 
there exists a bilinear form $\langle, \rangle : H_N \times H_N \mor \ZZ$
such that $\langle \schub_u, \schub_{u'w_0} \rangle = \delta_{uu'}$
and $\langle x^{\rho-\alpha}, \prod_{1 \leq i \leq n-1} e_{\beta_i}(x_1, \ldots, x_{n-i}) \rangle = \delta_{\alpha, \beta}$
where $\rho=(n-1, n-2, \ldots, 0)$. 
Then
\begin{align*}
n_{wu} &= \langle \schub_{uw_0}, \prod_{2\leq i \leq n} e_{l_i}(x_1, \ldots, x_{i-1}) \rangle \\
&= (\text{coefficient of $x_1^{n-1-l_n}x_2^{n-2-l_{n-1}}\cdots$ in $\schub_{uw_0}$}). 
\end{align*}
Here, $n-k-l_{n+1-k}=n-k-\#\{j<n+1-k : w(j)>w(n+1-k)\}=\#\{j<n+1-k : w(j)<w(n+1-k)\}
=\#\{j>k : ww_0(j)<ww_0(k)\}=\inv(ww_0)_k$, 
and thus the number $n_{wu}$ is equal to the coefficient of $x^{\inv(ww_0)}$ in $\schub_{uw_0}$,
which is nonzero only if $(ww_0)^{-1} \lex\geq (uw_0)^{-1}$, 
which is equivalent to $w^{-1} \lex\leq u^{-1}$. 
Moreover, if $u=w$ then we see that $n_{ww}=1$. 
Thus the subquotients of (any) KP filtration of $T_w$ 
are the modules $\smod_u$ ($u^{-1} \lex> w^{-1}$), 
together with $\smod_w$ which occurs only once. 
Since $\Ext^1(\smod_w, \smod_u)=0$ for $u^{-1} \lex> w^{-1}$, 
we can take the filtration to satisfy the additional condition that $\smod_w$ occurs as a submodule of $T_w$. 
This completes the proof of Lemma \ref{tillem}. 
\end{proof}

As we pointed in the previous paper (\cite[Proposition 8.3, Theorem 7.2]{W}), Theorem \ref{mainthm} has several interesting corollaries: 
\begin{cor}
\label{pleth}
Let $\lmb$ be a partition and let $s_\lmb$ be the Schur functor corresponding to $\lmb$. 
Then for any $w \in S_\infty^{(n)}$, 
$s_\lmb(\smod_w)$ has a KP filtration. 
In particular, if $\schub_w$ is written as a sum of monomials as $\schub_w=x^\alpha+x^\beta+\cdots$, 
then the polynomial $s_\lmb(x^\alpha, x^\beta, \cdots)$ (here $s_\lmb$ denotes the Schur function)
is a positive sum of Schubert polynomials. 
\end{cor}
\begin{cor}
Let $u, v, w \in S_n$. 
Then the coefficient of $\schub_w$ in the expansion of $\schub_u\schub_v$ into Schubert polynomials
is equal to the dimension of 
$\Hom_\borel(\smod_u \otimes \smod_v, \smod_{w_0w}^* \otimes K_\rho)
\cong \Hom_\borel(\smod_u \otimes \smod_v \otimes \smod_{w_0w}, K_\rho)$. 
\end{cor}

\bibliographystyle{plain}
%\bibliography{references}

\end{document}